\documentclass[11pt]{amsart}
\usepackage{amsmath}
\usepackage{graphicx}
\usepackage{epstopdf}
\usepackage{amssymb,amscd,array}
\usepackage{color}
\usepackage{float}

\makeindex \makeatletter
\def\captionof#1#2{{\def\@captype{#1}#2}}
\makeatother

\newcounter{tablegroup}
\newcounter{subtable}[tablegroup]

\newtheorem{thm}{Theorem}[section]
\newtheorem{cor}[thm]{Corollary}
\newtheorem{lem}[thm]{Lemma}

\newtheorem{rem}[thm]{\bf Remark}

\numberwithin{equation}{section}

\begin{document}
\title{Pointwise-recurrent dendrite maps}

\author{Issam Naghmouchi}
\address{Issam Naghmouchi, University of Carthage, Faculty of Sciences of Bizerte, Department of Mathematics,
Jarzouna, 7021, Tunisia.} \footnote{This work was supported by the
research unit 99UR/15-15}

\email{ issam.nagh@gmail.com}



\begin{abstract} Let $D$ be a dendrite and  $f:D\rightarrow D$ a continuous map. Denote by $E(D)$ and $B(D)$ the sets of endpoints and branch
points of $D$ respectively. We show that if $E(D)$ is countable
(resp. $B(D)$ is discrete) then $f$ is pointwise-recurrent if and
only if $f$ is pointwise periodic homeomorphism (resp. every point
in $D\backslash E(D)$ is periodic).
\end{abstract}
\maketitle

\section{\bf Introduction}
Recurrence and periodicity play an important role in studying
dynamical systems. It is interesting to study maps $f: X\rightarrow
X$ from a topological space $X$ to itself that are
pointwise-periodic (i.e. all points in $X$ are periodic), or
pointwise-recurrent (i.e. all points in $X$ are recurrent).
Montgomery \cite{mont} showed that if $X$ is a connected topological
manifold, a pointwise-periodic homeomorphism $f: X\rightarrow X$
must be periodic. Weaver \cite{w} showed that if $X$ is a continuum
embedded in an orientable $2$-manifold, an orientation-preserving
$C^{1}$-homeomorphism  $f: X\rightarrow X$ has this property.
Gottschalk \cite{Gott} proved that if $X$ is a continuum then
relatively recurrent homeomorphism $f: X\rightarrow X$(i.e. the
closure of recurrent point set is dense in $X$) has every recurrent
cut point periodic. In \cite{ov}, Oversteegen and Tymchatyn showed
that recurrent homeomorphisms of the plane are periodic. Kolev and
P$\acute{e}$rou$\grave{e}$me \cite{ko} proved that recurrent
homeomorphisms of a compact surface with negative Euler
characteristic are still periodic. Recently, Mai \cite{Ma1} showed
that a graph map $f: G\rightarrow G$ is pointwise-recurrent if and
only if one of the following statements holds:

(1) $G$ is a circle, and $f$ is a homeomorphism topologically
conjugate to an irrational rotation

(2) $f$ is a periodic homemorphism.

In this paper we will study pointwise-recurrent dendrite maps, their
dynamical behaviors are both important and interesting in the study
of dynamical systems and continuum theory. Recent interest in
dynamics on dendrites is motivated by the fact that dendrites have
often appear as Julia sets in complex dynamics (see \cite{Bea}). In
(\cite{Ac}, \cite{Bal}, \cite{Ef1}, \cite{Ma2}, \cite{Nagh1} and
\cite{Nagh}) several results concerning dendrites
 were obtained. In \cite{Nagh}, we proved that every relatively recurrent monotone dendrite map have all its cut points
periodic.  Before stating our main results, we recall some basic
properties of dendrites and dendrite maps.

A continuum is a compact connected metric space. A topological space
is arcwise connected if any two of its points can be joined by an
arc. We use the terminologies from Nadler \cite{Nadler}. An arc is
any space homeomorphic to the compact interval $[0,1]$. By a
\textit{dendrite} $D$, we mean a locally connected continuum which
contains no homeomorphic copy to a circle. Every sub-continuum of a
dendrite is a dendrite (\cite{Nadler}, Theorem 10.10) and every
connected subset of $D$ is arcwise connected (\cite{Nadler},
Proposition 10.9). In addition, any two distinct points $x,y$ of a
dendrite $D$ can be joined by a unique arc with endpoints $x$ and
$y$, denote this arc by $[x,y]$ and let denote by
$[x,y)=[x,y]\setminus\{y\}$ (resp. $(x,y]=[x,y]\setminus\{x\}$ and
$(x,y)=[x,y]\setminus\{x,y\}$). A point $x\in D$ is called an
\textit{endpoint} if $D\setminus\{x\}$ is connected. It is called a
\textit{branch point} if $D\setminus \{x\}$ has more than two
connected components. Denote by $E(D)$ and $B(D)$ the sets of
endpoints, and branch points of $D$ respectively. A point $x\in
D\setminus E(D)$ is called a \textit{cut point}. The set of cut
points of $D$ is dense in $D$. A tree is a dendrite with finite set
of endpoints.

Let $\mathbb{Z}_{+}$ and $\mathbb{N}$ be the sets of non-negative
integers and positive integers respectively. Let $X$ be a compact
metric space with metric $d$ and $f: X\longrightarrow X$ be a
continuous map. Denote by $f^{n}$ the $n$-th iterate of $f$; that
is,  $f^{0} =\textrm{id}_{X}$: the Identity and $f^{n} = f\circ
f^{n-1}$ if $n\geq 1$. For any $x\in X$ the subset $O_{f}(x)=
\{f^{n}(x): \ n\in\mathbb{Z}_{+}\}$ is called the $f$-orbit of $x$.
A point $x\in X$  is called periodic of prime period
$n\in\mathbb{N}$ if $f^{n}(x)=x$  and $f^{i}(x)\neq x$ for $1\leq
i\leq n-1$. A subset $A$ of $X$ is called \textit{$f$-invariant} if
$f(A)\subset A$. It is called \textit{a minimal set of $f$} if it is
non-empty, closed, $f$-invariant and minimal (in the sense of
inclusion) for these properties. For a subset $A$ of $X$, denote by
$\overline{A}$ the closure of $A$ and by diam($A)$ the diameter of
$A$. We define the $\omega$-limit set of a point $x$ to be the set
\begin{align*}
 \omega_{f}(x) & = \{y\in X: \exists\ n_{i}\in \mathbb{N},
n_{i}\rightarrow\infty, \lim_{i\rightarrow\infty}d(f^{n_{i}}(x), y)
= 0\}\\
& = \underset{n\in \mathbb{N}}\cap\overline{\{f^{k}(x): k\geq n\}}.
\end{align*}

The set $\omega_{f}(x)$ is a non-empty, closed and strongly
invariant set, i.e. $f(\omega_{f}(x))= \omega_{f}(x)$. A point $x\in
X$ is said to be:

-  \textit{recurrent} for $f$ if $x\in\omega_{f}(x)$.

- \textit{almost periodic} if for any neighborhood $U$ of $x$ there
exists $N\in\mathbb{N}$ such that $\{f^{n+i}(x): i=0,1,\dots,N\}\cap
U\neq\emptyset$ for all $n\in\mathbb{N}$.

- \textit{regularly recurrent} if for any $\varepsilon>0$, there is
$N\in\mathbb{N}$ such that $d(x,f^{kN}(x))<\varepsilon$ for all
$k\in\mathbb{N}$.

It is easy to see that if $x$ is regularly recurrent then it is
almost periodic, hence $\omega_{f}(x)$ is a minimal set (see
\cite{lsB}, Proposition 5, Chapter V).

Let Fix$(f)$, P($f)$, AP($f)$ and R($f)$ denote the set of fixed
points, periodic points, almost periodic and recurrent points
respectively. Then we have the following inclusion relation
$\rm{Fix}(f)\subset\rm{P}(f)\subset \rm{AP}(f)\subset \rm{R}(f)$.
  We say that $f$ is
\begin{itemize}
\item [-] \textit{pointwise-periodic} if P$(f)= X$.

\item [-] \textit{pointwise-recurrent} if R($f)= X$.

\item [-] \textit{relatively recurrent} if $\overline{\rm{R}(f)}= X$.
\end{itemize}
\medskip

A continuous map from a dendrite into itself is called a
\textit{dendrite map}.

Our main results are the following:

\begin{thm}\label{pr1} Let $f:~D\rightarrow D$ be a dendrite map. If B($D)$ is discrete then $f$ is
pointwise-recurrent if and only if $f$ is a homeomorphism and every
cut point is periodic.
\end{thm}
\bigskip

Following (\cite{Ar}, Corollary 3.6), for any dendrite $D$, we have
B($D)$ is discrete whenever E($D)$ is closed. Therefore:
\medskip

\begin{cor}\label{c1} Let $f:~D\rightarrow D$ be a dendrite map. If E($D)$ is closed then Theorem \ref{pr1} holds.
\end{cor}
\medskip

\begin{cor}\label{c33} If B($D)$ is discrete and $f:~D\rightarrow D$ is pointwise-recurrent dendrite map then every endpoint of
$D$ is regularly recurrent.
\end{cor}
\medskip

\begin{rem} \rm{If $E(D)$ is closed and $f:D\rightarrow D$ is
a pointwise-recurrent dendrite map, then an endpoint of $D$ may not
be in general periodic (see an example by Efremova and Makhrova in
\cite{Ef1} on Gehman dendrite).}\end{rem}
\medskip

\begin{thm}\label{pr2} Let $f:D\rightarrow D$ be a dendrite map. If E($D)$ is countable then $f$ is
pointwise-recurrent if and only if $f$ is pointwise-periodic
homeomorphism.
\end{thm}
\medskip

\begin{cor}\label{c2}\cite{Ma1} Let $T$ be a tree and $f: T\rightarrow T$ a continuous map. If $f$ is pointwise-recurrent then $f$ is periodic.
\end{cor}
\medskip

Recall that the map $f$ is periodic if $f^{n} = \textrm{id}_T$ for
some $n\in\mathbb{N}$.
\section{\bf Preliminaries}

\begin{lem}\label{l2}\rm{(\cite{Ma2}, Lemma 2.3)} Let $(C_{i})_{i\in\mathbb{N}}$ be a sequence of connected subsets of
a dendrite $(D,d)$. If $C_{i}\cap C_{j} = \emptyset$ for all $i\neq
j$, then $$\lim_{n\to +\infty}\mathrm{diam}(C_{n})=0.$$
\end{lem}
\medskip

\begin{lem}\label{den} Let $D$ be a dendrite and $(p_{n})_{n\in\mathbb{N}}$ be a sequence of $D$ such that
$p_{n+1}\in(p_{n},p_{n+2})$ for all $n\in\mathbb{N}$, and
$\underset{n\to +\infty}\lim p_{n}=p_{\infty}$. Let $U_{n}$ be the
connected component of $D\setminus\{p_{n},p_{\infty}\}$ that
contains the open arc $(p_{n},p_{\infty})$. Then $\underset{n\to
+\infty}\lim \mathrm{diam}(U_{n})=0$
\end{lem}
\medskip

\begin{proof} It is easy to see that $U_{n+1}\subset U_{n}$ for all $n\in\mathbb{N}$. Suppose that Lemma \ref{den} is not true, then there is $\delta>0$ such that for
all $n\in\mathbb{N}$, diam($U_{n})>\delta$. We will construct an
infinite sequence $(I_{n_{i}})_{i\in\mathbb{N}}$ of pairwise
disjoint arcs such that $diam(I_{n_{i}})\geq\frac{\delta}{3}$ for
all $i\in\mathbb{N}$ which contradicts Lemma \ref{l2}: Take
$n_{0}\in\mathbb{N}$ such that for all $n\geq n_{0}$,
$diam([p_{n},p_{\infty}])<\frac{\delta}{3}$. For an integer $n\geq
n_{0}$, let $a_{n},b_{n}\in U_{n}$ be such that
$d(a_{n},b_{n})>\delta$. There exist $c_{n},d_{n}\in
(p_{n},p_{\infty})$ such that $[a_{n},c_{n}]\cap
[p_{n},p_{\infty}]=\{c_{n}\}$ and $[b_{n},d_{n}]\cap
[p_{n},p_{\infty}]=\{d_{n}\}$. As
$d(p_{n},p_{\infty})<\frac{\delta}{3}$, we have either
$d(c_{n},a_{n})>\frac{\delta}{3}$ or
$d(d_{n},b_{n})>\frac{\delta}{3}$. So we let $I_{n}=[c_{n},a_{n}]$
if $d(c_{n},a_{n})>\frac{\delta}{3}$ and $I_{n}=[d_{n},b_{n}]$ if
$d(d_{n},b_{n})>\frac{\delta}{3}$. Choose an integer $m>n$ such that
$p_{m}\in (c_{n},p_{\infty})$ if $I_{n}=[a_{n},c_{n}]$ and $p_{m}\in
(d_{n},p_{\infty})$ if $I_{n}=[b_{n},d_{n}]$. Then $I_{n}\cap
U_{m}=\emptyset$: Indeed, otherwise there exists $z\in I_{n}\cap
U_{m}$. Take $I_{n}=[a_{n},c_{n}]$. As $c_{n}\notin U_{m}$ and $z\in
U_{}$, then $[z,c_{n}]$ contains $p_{n}$ or $p_{\infty}$. In either
cases, $[z,c_{n}]\supset [c_{n},p_{n}]\neq \{c_{n}\}$, a
contradiction. By the same way as $I_{n}$, we obtain an arc
$I_{m}\subset U_{m}$ such that $I_{m}$ intersect the arc
$[p_{m},p_{\infty}]$ in a single point and has diameter greater than
$\frac{\delta}{3}$. So by repeating this process infinitely many
times beginning from $n_{0}$, we obtain an infinite sequence of arcs
$(I_{n_{i}})_{i\in\mathbb{N}}$ with diameter greater than
$\frac{\delta}{3}$  and satisfying the following property: for all
$i\in\mathbb{N}$, $I_{n_{i}}\subset U_{n_{i}}$ and $I_{n_{i}}\cap
U_{n_{i+1}}=\emptyset$. This implies that
$(I_{n_{i}})_{i\in\mathbb{N}}$ are pairwise disjoint, which is our
claim.
\end{proof}
\medskip

\begin{thm}\label{rc}$($\cite{Ma2}, Theorem 2.13$)$ Let $f:D\rightarrow D$ be a dendrite map,
$[x,y]$ be an arc in $D$, and $U$ be the connected component of
$D\setminus \{x,y\}$ containing the open arc $(x, y)$. If there
exist $k,m\in\mathbb{N}$ such that $\{f^{k}(x), f^{m}(y)\}\subset
U$, then $U \cap P(f)\neq\emptyset$.
\end{thm}
\medskip

 We say that a dendrite map $f: D\rightarrow D$ is \textit{monotone} if the preimage of any point by $f$ is connected.
Notice that if $f$ is monotone so is $f^{n}$ for any
$n\in\mathbb{N}$.
\medskip

\begin{thm}$($\cite{Nagh}, Theorem 1.6$)$\label{nagh}
Let $f:D\rightarrow D$ be a dendrite map. If $f$ is monotone then
the following statements are equivalent:
\begin{itemize}
 \item [(i)] $f$ is pointwise-recurrent.
\item [(ii)] $f$ is relatively recurrent.
\item [(iii)] every cut point is a periodic point.
\end{itemize}
\end{thm}
\medskip

\begin{lem}\label{l33}
Let $f:D\rightarrow D$ be a dendrite map, then the following
properties are equivalent:
\begin{itemize}
 \item [(i)] any nondegenerate arc $I\subset f(D)$ contains a point with
unique preimage by $f$.
\item[(ii)] the map $f$ is monotone.
\end{itemize}
\end{lem}
\medskip

\begin{proof} $(i)\Longrightarrow (ii)$: If $f$ is not monotone, then there exists
$z\in D$ such that $f^{-1}(z)$ is not connected. So one can find $a,
b\in D$ with $a\neq b$ and  $w_{-1}\in (a,b)$ such that
$f(a)=f(b)=z$ and $w:= f(w_{-1})\neq z$. By continuity of $f$, we
have $[z,w]\subset f([a,w_{-1}])\cap f([b,w_{-1}])\subset f(D)$.
Hence each point in $(z,w)$ has at least two preimages by $f$, a
contradiction.

$(ii)\Longrightarrow (i)$:  Let $I\subset f(D)$ be a nondegenerate
arc. Then $(f^{-1}(\{x\}))_{x\in I}$ is a family of uncountably many
pairwise disjoint connected non-empty subsets of $D$. Suppose that
for every $x\in D$, $f^{-1}(\{x\})$ is not reduced to a point, then
there is a non degenerate arc $I_{x}\subset f^{-1}(\{x\})$
containing no endpoints. By (\cite{Nadler}, Corollary 10.28), $D$
can be written as follow: $D= \cup_{n\in\mathbb{N}}A_{n}\cup E(D)$
where $(A_{n})_{n\in\mathbb{N}}$ is a family of arcs with pairwise
disjoint interiors. Hence, for each arc $I_{x}$, there is
$n(x)\in\mathbb{N}$ such that $I_{x}\cap A_{n(x)}$ is a non
degenerate arc. So necessarily, there is an arc $A_{n_{0}}$
containing an uncountably many pairwise disjoint nondegenerate arcs
of $(I_{x})_{x\in I}$, which is a contradiction.
\end{proof}
\medskip

\begin{lem}\label{pree} Let $X$ be a compact metric space and $f:X\rightarrow X$
is a pointwise-recurrent continuous map. Then every periodic point
of $f$ has a unique pre-image by $f^{n}$ for all $n\in\mathbb{N}$.
\end{lem}
\medskip

\begin{proof} As $R(f) = R(f^{n})$ for all $n\in\mathbb{N}$ (see \cite{lsB}), it suffices to prove the Lemma for $f$.
Suppose that for some periodic point $p$ of period $n\in\mathbb{N}$,
$f^{-1}(\{p\})$ contains more than one point. So there is $q\neq
f^{n-1}(p)$ such that $f(q)=p$. Since $q$ is recurrent and
$\omega_{f}(q)= O_{f}(p)$, it follows that $q\in O_{f}(p)$, so there
is $k\in\mathbb{N}$ such that $q= f^{k}(p)$. Thus $f^{n-1}(p)=
f^{n}(q)= f^{k}(f^{n}(p)) = f^{k}(p)= q$, a contradiction.
     \end{proof}

\begin{lem}\label{l27} Let $f:D\rightarrow D$ be a dendrite map. Assume that $f$ is a pointwise-recurrent.
Let $(p_{n})_{n\in\mathbb{N}}\subset D$ be a sequence of periodic
points of $f$ such that $p_{n+1}\in(p_{n},p_{n+2})$ for all
$n\in\mathbb{N}$, and $\underset{n\rightarrow +\infty}\lim p_{n}=
p_{\infty}$. Then $p_{\infty}$ is a regularly recurrent point.
\end{lem}
\medskip

\begin{proof}
For $n\in \mathbb{N}$, let $U_{n}$ be defined as in Lemma \ref{den},
let $V_{n}:=U_{n}\cup \{p_{n},p_{\infty}\}$ and let denote by
$N_{n}$ the period of the point $p_{n}$. It is easy to see that if
there is a sub-sequence of $(p_{n})_{n\in\mathbb{N}}$ with bounded
periods then by the continuity of $f$ (and hence the continuity of
its iterated maps), the point $p_{\infty}$ is periodic, in
particular, it is regularly recurrent point. Otherwise, we have to
prove the Lemma 2.6 in the case of $p_{\infty}$ is not periodic.
Then without loss of generality, the sequence
$(p_{n})_{n\in\mathbb{N}}$ can be assumed such that
\begin{equation}\label{e2} (p_{n},p_{\infty}]\cap
\textrm{Fix}(f^{N_{n}}) = \emptyset.\end{equation}

We will prove that for all $n\in\mathbb{N}$, the orbit of the point
$p_{\infty}$ under the map $f^{N_{n+1}}$ is included into the set
$V_{n}$: Indeed, otherwise for some $n\in\mathbb{N}$, there is
$k\in\mathbb{N}$ such that $f^{kN_{n+1}}(p_{\infty})\notin V_{n}$,
so we have two possibilities both of them lead to a contradiction:
$p_{\infty}\in (p_{n+1},f^{kN_{n+1}}(p_{\infty}))$   or
 $p_{n}\in (p_{n+1},f^{kN_{n+1}}(p_{\infty}))$. Let $m:=kN_{n+1}$.

Suppose that $p_{\infty}\in (p_{n+1},f^{m}(p_{\infty}))$. As
$p_{n+1}\in\rm{Fix}(f^{m})$, by the continuity of $f^{m}$, we have
$f^{m}([p_{n+1},p_{\infty}])\supset [p_{n+1},f^{m}(p_{\infty})]\ni
p_{\infty}$. Hence there is a point $p_{\infty,-1}\in
(p_{n+1},p_{\infty})$ such that $f^{m}(p_{\infty,-1})=p_{\infty}$.
Similarly, there is a point $p_{\infty,-2}\in
(p_{n+1},p_{\infty,-1})$ such that
$f^{m}(p_{\infty,-2})=p_{\infty,-1}$. Thus, by induction we
construct a sequence $(p_{\infty,-k})_{k\in\mathbb{N}}$ in
$(p_{n+1},p_{\infty})$ such that for all $k\in\mathbb{N}$,
$p_{\infty,-(k+1)}\in[p_{n+1},p_{\infty,-k}]$  and
\begin{equation}\label{e1}
f^{km}(p_{\infty,-k})=p_{\infty}.
\end{equation}
As $p_{\infty,-1}\neq p_{\infty}$, there is
$p_{s}\in(p_{\infty,-1},p_{\infty})$ for some $s\in\mathbb{N}$. Let
$r\in\mathbb{N}$ be such that $f^{r}(p_{n+1})=p_{n+1}$ and
$f^{r}(p_{s})=p_{s}$. Hence, $f^{rm}(p_{n+1})=p_{n+1}$ and
$f^{rm}(p_{s})=p_{s}$. By (\ref{e1}),
$f^{rm}(p_{\infty,-r})=p_{\infty}$ and by the continuity of
$f^{rm}$, $f^{rm}([p_{n+1},p_{\infty,-r}])\supset
[p_{n+1},p_{\infty}]\ni p_{s}$, hence $p_{s}$ has a pre-image $q$ by
the map $f^{rm}$ in the arc $[p_{n+1},p_{\infty,-r}]$ and as
$p_{s}\notin [p_{n+1},p_{\infty,-r}]$, $q\neq p_{s}$, this
contradicts Lemma \ref{pree} since $p_{s}$ is a fixed point of
$f^{rm}$.

Suppose now the second case, $p_{n}\in (p_{n+1},f^{m}(p_{\infty}))$.
By the continuity of $f^{m}$, $f^{m}([p_{n+1},p_{\infty}])\supset
[p_{n+1},f^{m}(p_{\infty})]\ni p_{n}$. So $p_{n}$ has a preimage $q$
by the map $f^{m}$ in the arc $[p_{n+1},p_{\infty}]$. By Lemma
\ref{pree}, $q$ is a periodic point that belongs to the orbit of
$p_{n}$ hence $q$ has the same period as $p_{n}$. Hence,
$q\in(p_{n},p_{\infty}]\cap Fix(f^{N_{n}})$, this contradict
(\ref{e2}).

It follows that for any $n\in\mathbb{N}$, the orbit of the point
$p_{\infty}$ under the map $f^{N_{n+1}}$ is included into the set
$V_{n}$ and as diam($V_{n})=\textrm{diam}(U_{n})$,\\ $\underset{n\to
+\infty}\lim \textrm{diam}(V_{n})=0$, by Lemma 2.5. This implies
that $p_{\infty}$ is regularly recurrent point.
\end{proof}
\medskip

\section{\bf Proof of Theorem \ref{pr1} and Corollary \ref{c33}}
\medskip

\begin{proof}[Proof of Theorem \ref{pr1}:]
The ``if'' part of the Theorem results clearly from Theorem
\ref{nagh} since $f$ is in particular monotone. Lets prove the
``only if'' part:

Assume that $f$ is pointwise-recurrent, then $f$ is surjective. We
will use Lemma \ref{l33}: Let $I\subset D$ be a nondegenerate arc.
Since $B(D)$ is discrete, there exists a non-degenerate open arc
$J\subset I$ containing no branch points, hence $J$ is an open
subset in $D$. So let $x\in J$. Since $f$ is pointwise-recurrent,
one can find $n, m\in\mathbb{N}$ such that $x,f^{m}(x)\in J$ and
$f^{n+m}(x)\in(x,f^{m}(x))$. Thus $(x,f^{m}(x))\subset I$ is the
connected component of $D\setminus\{x,f^{m}(x)\}$ containing the
open arc $(x,f^{m}(x))$. By Theorem \ref{rc}, $(x,f^{m}(x))$
contains a periodic point $p$ and by Lemma \ref{pree}, $p$ has a
unique preimage by $f$. We conclude by Lemma \ref{l33} that $f$ is
monotone. Therefore, by (\cite{Nagh}, Corollary 1.7), $f$ is a
homeomorphism and by Theorem \ref{nagh}, every cut point is
periodic. The proof is complete.
\end{proof}
\medskip

\begin{proof}[Proof of Corollary \ref{c33}:] Let $e\in E(D)$. By Theorem \ref{pr1}, one can find a sequence of periodic points
$(p_{n})_{n\in\mathbb{N}}\subset D$ such that
$p_{n+1}\in(p_{n},p_{n+2})$ for all $n\in\mathbb{N}$ and
$\underset{n\rightarrow +\infty}\lim p_{n} = e$. Hence by Lemma
\ref{l27}, $e$ is regularly recurrent.
\end{proof}
\medskip

\section{\bf Proof of Theorem \ref{pr2}}
We need the following result.

\begin{thm}\cite{Ma2}\label{t41} Let $f: D\rightarrow D$ be a dendrite map. If E($D)$ is countable then
$\overline{R(f)} = \overline{P(f)}$.
\end{thm}
\medskip

\begin{lem}\label{l26} Let $f: D\rightarrow D$ be a dendrite map. If E($D)$ is countable and $f$ is pointwise-recurrent then $\omega_{f}(x)\cap
P(f)\neq\emptyset$ for all $x\in D$.
\end{lem}
\medskip

\begin{proof} If $\omega_{f}(x)\subset E(D)$ then
$\omega_{f}(x)$ is compact countable. If $\omega_{f}(x)$ is infinite
then it is perfect. Since for any $y\in\omega_{f}(x)$,
$y=\underset{k\rightarrow +\infty}\lim f^{n_k}(x)$ where
$(n_k)_{k\in\mathbb{N}}$ is an infinite sequence of positif
integers. As $x\in\omega_{f}(x)$, then $O_f(x)\subset \omega_f(x)$,
hence $y$ is an accumulation point of $\omega_{f}(x)$. So
$\omega_{f}(x)$ is uncountable, a contradiction. Therefore,
$\omega_{f}(x)$ is finite, that is $\omega_{f}(x)$ is a periodic
orbit and so $x$ is a periodic point. One can then assume that
$\omega_{f}(x)\backslash E(D)\neq \emptyset$, so let
$y\in\omega_{f}(x)\backslash E(D)$. Since $R(f) = D$, it follows by
Theorem \ref{t41} that $\overline{P(f)}=D$, so one can find $p,q\in
P(f)$ such that $y\in [p,q]$. Let $N\in\mathbb{N}$ be such that
$f^{N}(p)=p$ and $f^{N}(q)=q$. We already have  $y\in [p,
f^{N}(y)]\cup [q,f^{n}(y)]$. One can suppose that  $y\in [p,
f^{N}(y)]$, the same proof being true with $p$ replaced by $q$. In
this case, we have $y\in f^{N}([p,y])$, so there is $y_{-1}\in
[p,y]$ such that $f^{N}(y_{-1})=y$. Again, $y_{-1}\in [p,y]=
[p,f^{N}(y_{-1})]\subset f^{N}([p,y_{-1}])$, so there is $y_{-2}\in
[p,y_{-1}]$ such that $f^{N}(y_{-2})=y_{-1}$. Thus, we construct by
induction a sequence $(y_{-n})_{n\in\mathbb{N}}\subset [p,y]$ such
that $y_{-(k+1)}\in (y_{-k}, y_{-(k+2)})$,
$f^{N}(y_{-(k+1)})=y_{-k}$ for every $k\in\mathbb{N}$, and
$\underset{n\to +\infty}\lim y_{-n}= y_{\infty}\in [p,y]$. Therefore
one has $f^{N}(y_{\infty})= y_{\infty}$ (by the continuity of
$f^{N}$) and so $y_{\infty}\in P(f)$. Since $R(f)=D$ and for any
$k\in\mathbb{N}$, $f^{kN}(y_{-k})=y$, it follows that $y_{-k}\in
\omega_{f}(y_{-k}) = \omega_{f}(y)$. Therefore
$y_{\infty}\in\omega_{f}(y)$ and as $\omega_{f}(y)\subset
\omega_{f}(x)$, we conclude that $y_{\infty}\in\omega_{f}(x)\cap
P(f)$.
\end{proof}
\medskip

\begin{lem}\label{l43} Let $f:D\rightarrow D$ be a dendrite map and let $p_{\infty}$ as in Lemma \ref{l27}. If E($D)$ is countable and $f$ is pointwise-recurrent, then
$p_{\infty}$ is a periodic point.
\end{lem}
\medskip

\begin{proof} By Lemma \ref{l27}, $p_{\infty}$ is regularly recurrent, hence $\omega_{f}(p_{\infty})$ is a minimal set. It follows, by Lemma \ref{l26}, that $\omega_{f}(p_{\infty})$ is a
periodic orbit, and so $p_{\infty}$ is a periodic point since
$p_{\infty}\in \omega_{f}(p_{\infty})$.
 \end{proof}
\medskip

\begin{lem} Let $D$ be a dendrite with countable set of endpoints. Then every sub-dendrite of $D$ has countable set of endpoints.
\end{lem}
\medskip

\begin{proof} Let $Y$ be a sub-dendrite of $D$. By (\cite{go}, page 157, (see also
\cite{Nadler})), $Y$ is a monotone retraction of $D$ by
$r:D\rightarrow Y$. Then we have $r(E(D))=E(Y)$. Indeed,
$r(E(D))\subset E(Y)$ follows from the fact that $r(x)$ lies in any
arc joining $x$ to any point of $Y$. Let $a\in E(Y)$ and let $Z$ the
connected component of $D\setminus Y$ such that $\overline{Z}$
contains $a$. Then $r(\overline{Z})=\{a\}$. As $\overline{Z}\cap
E(D)\neq \emptyset$, there is $b\in \overline{Z} \cap E(D)$ with
$r(b)=a$. It follows that $Card(E(Y))\leq Card(E(D))$ and hence
$E(Y)$ is countable.
\end{proof}

\begin{proof}[Proof of Theorem \ref{pr2}:] The ``if'' part of the Theorem is clear. Lets prove the ``only
if'' part: By Theorem \ref{t41}, one has $\overline{P(f)}= D$.

\textit{Claim .1.} If $x\in D$ is not periodic then for any $z\in D$
with
$z\neq x$, $(x,z)$ contains a non periodic branch point.\\
Indeed, $[x,z]$ contains an infinite sequence of branch points that
converges to $x$. Since otherwise, there is $y\in (x,z]$ such that
$(x,y)\cap B(D)\neq \emptyset$. So $(x,y)$ is an open subset of $D$,
hence as $\overline{P(f)}=D$, one can find a sequence
$(p_{n})_{n\in\mathbb{N}}$ of periodic points in $(x,y)$ such that
$p_{n+1}\in(p_{n},p_{n+2})$ for all $n\in\mathbb{N}$, and
$\underset{n\to +\infty}\lim p_{n}= x$. By Lemma \ref{l43}, $x$ is
periodic, a contradiction. Now, Suppose that every branch point in
$(x,z)$ is periodic, then again by Lemma \ref{l43}, $x$ is periodic.
Therefore, necessarily one branch point in $(x,z)$ must be not
periodic.

\textit{Claim.2.} If there exists $x\in D$ not periodic then $D$
contains a sub-dendrite with uncountably set of endpoints.\\
Suppose that $x\in D$ is not periodic. By Claim.1, there is $a\in
B(D)$ not periodic. Now, as $a$ is a branch point, there exist two
nondegenerate arcs $I_{0}$, $I_{1}$ with one endpoint is $x$ and
form with $[x,a]$ a family of three arcs having pairwise disjoint
interior. As $a$ is not periodic then by Claim.1, there is $a_{0}\in
I_{0}$ (resp. $a_{1}\in I_{1}$) a non periodic branch point distinct
from $a$. Denote by $J_{0}$ (resp. $J_{1}$) the arc $[a,a_{0}]$
(resp. $[a,a_{1}]$). In the second step, as $a_{0}$ and $a_{1}$ are
branch points, there exist two nondegenerate arcs $I_{00}$ and
$I_{01}$ (resp. $I_{10}$ and $I_{11}$) with one endpoint is $a_{0}$
(resp. $a_1$) and form with $[a,a_0]$ (resp. $[a,a_1]$) a family of
three arcs having pairwise disjoint interior. Similarly, by Claim.1,
we find, for any $i,j\in\{0,1\}$, a non periodic branch point
$a_{ij}\in I_{ij}$ distinct from $a_{i}$, let denote by
$J_{ij}:=[a_{i},a_{ij}]$. Thus by induction, we construct for all
$n\in\mathbb{N}$ a sequence of arcs $J_{\alpha_n}$ and a sequence of
non periodic branch points $a_{\alpha_n}$ where
$\alpha_n\in\{0,1\}^{n}$ satisfying the following properties:\\
(i) $J_{\alpha_n 0}=[a_{\alpha_n},a_{\alpha_n 0}]$ and $J_{\alpha_n
1}=[a_{\alpha_n},a_{\alpha_n 1}]$, \\
(ii) $J_{\alpha_n}\cap J_{\alpha_n 0}=J_{\alpha_n}\cap J_{\alpha_n
1}= J_{\alpha_n 0}\cap J_{\alpha_n 1}=\{a_{\alpha_n}\}$.
 any $n\in\mathbb{N}$,
$D_{n}=\cup_{\alpha_n\in\{0,1\}^{n}}J_{\alpha_n}$, then clearly
$D_{n}$ is a tree. Also we have $D_{n+1}\subset D_{n}$ for all
$n\in\mathbb{N}$. Hence
$D_{\infty}=\overline{\cup_{n\in\mathbb{N}}D_{n}}$ is a
sub-dendrite. Take $\beta\in\{0,1\}^{\mathbb{N}}$. For each
$n\in\mathbb{N}$, we have $[a,a_{\beta(n)}]\subset
[a,a_{\beta(n+1)}]$ (where for $i\in\mathbb{N}$, we denote by
$\beta(i)$ the first word of length $i$ in the sequence $\beta$).
Then $\overline{\cup_{n\in\mathbb{N}}[a,a_{\beta(n)}]}$ is an arc
where the sequence $(a_{\beta(n)})_{n\in\mathbb{N}}$ is monotone in
that arc so it converges to a point namely $a_{\beta}$ that belongs
to the dendrite $D_{\infty}$. The points $a_{\beta}$ for
$\beta\in\{0,1\}^{\mathbb{N}}$ are the endpoints of the sub-dendrite
$D_{\infty}$. Hence,  $D_{\infty}$ has an uncountable set of
endpoints .

We deduce then Theorem \ref{pr2} from Claim.2 and Lemma 4.4.
\end{proof}
\medskip
\begin{proof}[Proof of Corollary 1.6:] By Theorem \ref{pr2}, $f$ is
pointwise-periodic homeomorphism. As $E(T)$ is finite, there is
$N\in \mathbb{N}$ such that every endpoint of $T$ is fixed by $f^N$.
Let $a,b\in T$ be two distinct endpoints of $T$. As $f$ is a
homeomorphism (and so is $f^{N}$), the arc $[a,b]$ is
$f^{N}$-invariant, hence every point in $[a,b]$ is fixed by $f^{N}$.
Since $T=\cup_{a,b\in E(T)} [a,b]$, every point in the tree $T$ is
fixed by $f^{N}$. So $f$ is periodic.
\end{proof}

\textbf{Acknowledgements.} I would like to thanks Professor Habib
Marzougui for helpful discussions on the subject of the paper.
\bibliographystyle{amsplain}

\end{document}